\documentclass[a4paper, 11pt, reqno]{amsart}%

\usepackage{amssymb}      
\usepackage{amsmath} 
\usepackage{amsthm}
\usepackage{amsfonts}         
     
\usepackage{graphicx} 

\theoremstyle{plain}
\newtheorem{theorem}{Theorem}[section]

\newtheorem{cor}[theorem]{Corollary}

\newtheorem{lem}[theorem]{Lemma}

\newtheorem{prop}[theorem]{Proposition}

\newtheorem{mtheorem}{Theorem} 

\theoremstyle{definition}

\newtheorem{rem}[theorem]{Remark}
\newtheorem{mrem}{Remark}

\numberwithin{equation}{section}
\numberwithin{theorem}{section}

\allowdisplaybreaks[4] 

\begin{document} 

\title{Logistic elliptic equation with a nonlinear boundary condition arising from coastal fishery harvesting II}

\author{Kenichiro Umezu}  

\address{Department of Mathematics, Faculty of Education, Ibaraki University, Mito 310-8512, Japan}

\email{\tt kenichiro.umezu.math@vc.ibaraki.ac.jp} 



\thanks{} 



    \maketitle

\begin{abstract}
We study positive solutions of a logistic elliptic equation with a nonlinear boundary condition that models coastal fishery harvesting (\cite{GUU19}). An essential role is played by the smallest eigenvalue of the Dirichlet eigenvalue problem, in terms of which, a noncritical case is studied in \cite{Um2022}. In this paper, we extend our analysis to the {\it critical} case. We also further study the noncritical case for a more precise description of the positive solution set, including uniqueness and stability analysis for {\it large parameters}. Our approach relies on an energy method, sub- and supersolutions, and implicit function analysis. 
\end{abstract}

\section{Introduction} 

This paper is devoted to the study of positive solutions for the following logistic elliptic equation with a nonlinear boundary condition arising from coastal fishery harvesting (\cite{GUU19}):
\begin{align} \label{p}
\begin{cases}
-\Delta u = u-u^{p} & \mbox{ in } \Omega, \\
u\geq0 & \mbox{ in } \Omega, \\ 
\frac{\partial u}{\partial \nu}  = -\lambda u^q & \mbox{ on } \partial\Omega. 
\end{cases} 
\end{align}
Here, $\Omega\subset \mathbb{R}^N$, $N\geq1$, is a bounded domain with smooth boundary $\partial\Omega$, 
$\Delta = \sum_{i=1}^N \frac{\partial^2}{\partial x_i^2}$ is the usual Laplacian in $\mathbb{R}^N$, $0<q<1<p$, $\lambda \geq 0$ is a parameter, and 
$\nu$ is the unit outer normal to $\partial\Omega$. Unless stated otherwise, throughout this paper we assume the subcritical condition 
\begin{align} \label{subcp}
p<\frac{N+2}{N-2} \quad\mbox{ for } \ N>2. 
\end{align}
In the case of $p=2$, the unknown function $u\geq0$ ecologically represents the biomass of fish that inhabit a \textit{lake} $\Omega$, obeying the logistic law (\cite{CCbook}), and the nonlinear boundary condition 
means fishery harvesting with the harvesting effort $\lambda$ on the \textit{lake coast} $\partial\Omega$, obeying the Cobb--Douglas production function (\cite[Subsection 2.1]{GUU19}).

A nonnegative function $u\in H^1(\Omega)$ is called a \textit{nonnegative (weak) solution} of \eqref{p} if $u$ satisfies  
\begin{align} \label{def}
\int_{\Omega} \biggl( \nabla u \nabla \varphi -u\varphi + u^p \varphi \biggr) + \lambda \int_{\partial\Omega} u^{q} \varphi=0, \quad \varphi \in H^1(\Omega)
\end{align}
(we may regard $(\lambda,u)$ as a nonnegative solution of \eqref{p}). 
It is seen that problem \eqref{p} has a solution $(\lambda,0)$ for every $\lambda>0$, called a {\it trivial solution}. The sets $\{ (\lambda,0) : \lambda\geq0\}$ 
and $\{ (\lambda, 0) : \lambda > 0 \}$ are said to be the {\it trivial lines}. We know 
(\cite{Ro2005}) that a nonnegative solution $u$ of \eqref{p} belongs to the space $W^{1,r}(\Omega)$ for $r>N$ (consequently, $C^{\theta}(\overline{\Omega})$ for  $\theta \in (0,1)$). 
Moreover, a nontrivial nonnegative solution $u$ of \eqref{p} satisfies that $u\in C^{2+\theta}(\Omega)$ for $\theta \in (0,1)$, and $u>0$ in $\Omega$ (\cite{GT83}, \cite{PW67}), which is called \textit{a positive solution}. Indeed, 
if $u>0$ in $\overline{\Omega}$, then $u\in C^{2+\theta}(\overline{\Omega})$ by the bootstrap argument using elliptic regularity, and $u$ satisfies \eqref{p} pointwisely in $\overline{\Omega}$ in the classical sense. 
However, we do not know if $u>0$ on the entirety of $\partial\Omega$ for a positive solution $u$ of \eqref{p}. As a matter of fact, Hopf's boundary point lemma (\cite{PW67}) does not work because of the lack of the {\it one-sided Lipschitz condition} \cite[(4.1.19)]{Pa92} for mapping $0\leq u \mapsto (-u^q)$ for $u$ close to $0$.

For a positive solution $(\lambda,u)$ of \eqref{p} satisfying $u>0$ in $\overline{\Omega}$, we call $\gamma_1=\gamma_1(\lambda,u) \in \mathbb{R}$ the smallest eigenvalue of the linearized eigenvalue problem at $(\lambda,u)$
\begin{align} \label{gam0956}
\begin{cases}
-\Delta \varphi = \varphi - pu^{p-1}\varphi + \gamma \varphi & \mbox{ in } \Omega, \\
\frac{\partial\varphi}{\partial \nu} = - \lambda q u^{q-1} \varphi + \gamma \varphi & \mbox{ on } \partial\Omega.
\end{cases}    
\end{align}
It is well known that $\gamma_1$ is simple with a positive eigenfunction $\varphi_1\in C^{2+\theta}(\overline{\Omega})$ satisfying  $\varphi_1>0$ in $\overline{\Omega}$. 
Indeed, $\gamma_1$ is characterized by the variational formula
\begin{align*}
\gamma_1 
= \inf\left\{ \int_\Omega \biggl( |\nabla \varphi|^2 - \varphi^2 + pu^{p-1} \varphi^2 \biggr) + \lambda \int_{\partial \Omega} q u^{q-1} \varphi^2 : \varphi \in H^1(\Omega), \ \int_\Omega \varphi^2 + \int_{\partial\Omega} \varphi^2 = 1 \right\}. 
\end{align*}
A positive solution $u>0$ in $\overline{\Omega}$ of \eqref{p} is said to be {\it asymptotically stable}, {\it weakly stable}, and {\it unstable} if $\gamma_1>0$, 
$\gamma_1\geq 0$, and $\gamma_1<0$, respectively.

Problem \eqref{p} possesses a sublinear nonlinearity at infinity and also a concave--convex nature. Thus, the global uniqueness of a positive solution of \eqref{p} for every $\lambda>0$ would not be so easy to deduce. 
For nonlinear elliptic problems with a concave--convex nature, we refer to \cite{ABC94, Ta95, Al99, ACP01, FGU03, FGU06, Ko13}. 
The sublinear nonlinearity $(-u^q)$ that appears in \eqref{p} 
induces the absorption effect on $\partial\Omega$. 
Sublinear boundary conditions of the $u^q$ type were explored in 
\cite{GA04, imcom2007, GM08, RQU2014, RQU2015}. 
The case of an incoming flux on $\partial\Omega$ was studied in \cite{GA04, GM08}. 
The mixed case of absorption and an incoming flux on $\partial\Omega$ was studied in \cite{imcom2007}. The absorption case was also studied in \cite{RQU2014, RQU2015}, where a similar type of logistic elliptic equation with an indefinite weight has been analyzed for the existence and multiplicity of nontrivial nonnegative solutions.

An important role is played by the smallest eigenvalue $\beta_\Omega>0$ of the Dirichlet eigenvalue problem
\begin{align*} 
\begin{cases}
-\Delta \phi = \beta \phi & \mbox{ in } \Omega, \\
\phi = 0 & \mbox{ on } \partial \Omega. 
\end{cases}    
\end{align*} 
It is well known that $\beta_{\Omega}$ is simple with a positive eigenfunction 
$\phi_{\Omega}\in H^1_0(\Omega)$ 
(implying $\phi_\Omega \in C^{2+\theta}(\overline{\Omega})$ by elliptic regularity). 
Indeed, $\phi_\Omega > 0$ in $\Omega$, and 
\begin{align} \label{bdry1513}
c_{1}\leq -\frac{\partial \phi_{\Omega}}{\partial \nu}\leq c_{2} \quad\mbox{ on } \partial \Omega 
\end{align}
for some $0<c_1<c_2$. Moreover, $\beta_{\Omega}$ is characterized by the variational formula
\begin{align*} 
\beta_{\Omega} = \inf\left\{ \int_{\Omega}|\nabla \phi|^2 : \phi\in H^1_0(\Omega), \ \int_{\Omega}\phi^2=1 \right\}. 
\end{align*}
If $\beta_\Omega<1$, then $u_{\mathcal{D}}\in H^1_0(\Omega)\cap C^{2+\theta}(\overline{\Omega})$ denotes the unique positive solution of the Dirichlet logistic problem (\cite{BO86})
\begin{align} \label{Dp}
\begin{cases}
-\Delta u = u-u^{p} & \mbox{ in } \Omega, \\
u=0 & \mbox{ on } \partial\Omega. 
\end{cases}
\end{align}
The existence, nonexistence, and multiplicity of positive solutions for \eqref{p} in the case where $\beta_\Omega \neq 1$ were studied in the author's previous work \cite[Theorems 1.1, 1.2, 1.4, 1.5]{Um2022}, which Theorem \ref{thm0} summarizes. 

%
%

\setcounter{mtheorem}{-1}
\begin{mtheorem}  \label{thm0}
\strut 
\begin{enumerate} \setlength{\itemsep}{0.2cm} 
\item[(I)] A positive solution $u$ of \eqref{p} satisfies that $u<1$ in $\overline{\Omega}$ and $u>0$ on $\Gamma \subset \partial\Omega$ with the condition $|\Gamma|>0$.  

\item[(II)] There exists $\lambda_{\ast}>0$ such that problem \eqref{p} has a positive solution curve 
\begin{align} \label{C_0}
\mathcal{C}_0 = \{ (\lambda,u_\lambda) : 0\leq \lambda\leq \lambda_{\ast} \},
\end{align}
emanating from $(\lambda,u)=(0,1)$, that satisfies the following {\rm three} conditions:
\begin{itemize} \setlength{\itemsep}{0.1cm} 

\item $\lambda \mapsto u_\lambda \in C^{2+\theta}(\overline{\Omega})$ is $C^\infty$,

\item $u_\lambda > 0$ in $\overline{\Omega}$, 

\item $u_\lambda$ is asymptotically stable.

\end{itemize}
Moreover, the positive solutions of \eqref{p} near $(\lambda,u)=(0,1)$ in $\mathbb{R}\times C^{2+\theta}(\overline{\Omega})$ form $\mathcal{C}_0$. 

%
\end{enumerate}
Let $\overline{\lambda}$ be the positive value defined as 
\begin{align}  \label{ovlam}
\overline{\lambda} = \sup\{ \lambda > 0 : \mbox{ \eqref{p} has a positive solution for $\lambda$} \}. 
\end{align}
Then, the following assertions hold.

\begin{enumerate} \setlength{\itemsep}{0.2cm} 

\item[(III)] Assume that $\beta_\Omega < 1$. Then, we have the following (as in Figure \ref{figbetaleq1b}). 

\begin{enumerate} \setlength{\itemsep}{0.1cm} 
\item[(i)] $\overline{\lambda}=\infty$, and more precisely, problem \eqref{p} possesses 
a positive solution $u$ for every $\lambda>0$ such that $u>0$ in $\overline{\Omega}$.

\item[(ii)] $(\lambda, u_\lambda) \in \mathcal{C}_0$ is a unique positive solution of \eqref{p} for $0<\lambda\leq \lambda_{\ast}$ (by making  $\lambda_{\ast}$ in \eqref{C_0} smaller if necessary).  

\item[(iii)] $u_n \rightarrow u_{\mathcal{D}}$ in $H^1(\Omega)$ for a positive solution $(\lambda_n,u_n)$ of \eqref{p} with $\lambda_n\to \infty$.

\item[(iv)] The positive solution set $\{ (\lambda,u) \}$ does not meet the trivial line $\{ (\lambda, 0) : \lambda\geq0 \}$ in the topology of $H^1(\Omega)$ (nor $C(\overline{\Omega})$). 

\end{enumerate}

\item[(IV)] Assume that $\beta_\Omega>1$. Then, we have the following (as in Figure \ref{figsuperc}). 

\begin{enumerate} \setlength{\itemsep}{0.1cm} 

\item[(i)] $\overline{\lambda}<\infty$. 

\item[(ii)] There exists a bounded {\rm subcontinuum} (closed connected subset) $\widetilde{\mathcal{C}_0}=\{(\lambda,u)\}$ of nonnegative solutions of \eqref{p} in $[0, \infty)\times C(\overline{\Omega})$ joining $(\lambda,u)=(0,1)$ and $(0,0)$ such that $\widetilde{\mathcal{C}_0}\setminus \{ (0,0)\}$ includes $\mathcal{C}_0$ and consists of positive solutions of \eqref{p}. 
Particularly, problem \eqref{p} has at least two positive solutions for $\lambda>0$ small.

\item[(iii)] The positive solution set $\{ (\lambda,u)\}$ does not meet the trivial line $\{ (\lambda, 0) : \lambda>0 \}$ in the topology of $H^1(\Omega)$ (nor $C(\overline{\Omega})$). 

\item[(iv)] $\gamma_1(\lambda_n,u_n)<0$ for a positive solution $(\lambda_n,u_n)$ of \eqref{p} such that $(\lambda_n,u_n)\rightarrow (0,0)$ in $\mathbb{R}\times H^1(\Omega)$, provided that $u_n>0$ in $\overline{\Omega}$, i.e., $u_n$ is unstable.

\end{enumerate}
\end{enumerate}
\end{mtheorem}

%
%

\setcounter{mrem}{-1}
\begin{mrem}  \label{rem0}
\strut 
\begin{enumerate} \setlength{\itemsep}{0.1cm} 

\item Assertions (I) and (II) hold for every case of $\beta_\Omega > 0$. 

\item Assertions (II) and (III-i) hold for any $p>1$. 


\item Problem \eqref{p} with $\lambda = 0$ has exactly two nonnegative solutions $(\lambda,u)=(0,0), (0,1)$. Thus, Theorem \ref{thm0}(I) is used to show easily that 
in every case of $\beta_\Omega > 0$, the positive solution set $\{ (\lambda,u)\}$ of \eqref{p} meets {\it at most} $(0,0)$ and $(0,1)$ on $\{ (0,u) : u\geq0 \}$, i.e., if $(\lambda_n,u_n)$ is a positive solution of \eqref{p} such that $(\lambda_n,u_n) \rightarrow (0,u)$ in $H^1(\Omega)$ (equivalently $C(\overline{\Omega})$ by elliptic regularity), then either $u=0$ or $1$. 

\end{enumerate} 
\end{mrem} 

In this paper, we extend our consideration to the case where $\beta_\Omega=1$ and further study the positive solution set in the case where $\beta_\Omega < 1$. 
Our first main result concerns the case where $\beta_\Omega<1$. 
On the basis of Theorem \ref{thm0}(III), we present the uniqueness and stability of a positive solution of \eqref{p} for $\lambda > 0$ {\it large} and also the {\it strong positivity} of the positive solutions for {\it every} $\lambda>0$. 

\begin{theorem} \label{thm:leq1}
Assume that $\beta_\Omega<1$. Then, 
the following assertions hold (see Figure \ref{figbetaleq1b}): 

\begin{enumerate} \setlength{\itemsep}{0.2cm} 
\item There exists $\lambda^{\ast}\geq \lambda_{\ast}$ such that the positive solution of \eqref{p} ensured by Theorem \ref{thm0}(III-i) is {\rm unique} for every $\lambda>\lambda^{\ast}$ (say $u_\lambda$); more precisely, the positive solutions of \eqref{p} for $\lambda>\lambda^{\ast}$ form a $C^\infty$ curve $\mathcal{C}_\infty = \{ (\lambda,u_\lambda) : \lambda^{\ast}<\lambda \}$ (i.e., $\lambda \mapsto u_\lambda \in C^{2+\theta}(\overline{\Omega})$ is $C^\infty$), which satisfies the following conditions: 

\vspace{3pt}

\begin{enumerate} \setlength{\itemsep}{0.1cm} 

\item $u_\lambda$ is asymptotically stable,

\item $u_\lambda \longrightarrow u_{\mathcal{D}}$ in $H^1(\Omega)$ as $\lambda \to \infty$,

\item $u_\lambda$ is decreasing, i.e., $u_{\lambda_1}>u_{\lambda_2}$ in $\overline{\Omega}$ if $\lambda_1<\lambda_2$. Furthermore, if $0<\lambda_1<\lambda_2$ with the condition that $\lambda_1\leq \lambda^{\ast}<\lambda_2$, then $u>u_{\lambda_2}$ in $\overline{\Omega}$ for a positive solution $u$ of \eqref{p} for $\lambda=\lambda_1$. 
\end{enumerate}

\item $u>0$ in $\overline{\Omega}$ for a positive solution $u$ of \eqref{p} for every $\lambda>0$ {\rm (strong positivity)}.  

\end{enumerate} 

\end{theorem}

		 \begin{figure}[!htb]
	  	   \begin{center}
			\includegraphics[scale=0.18]{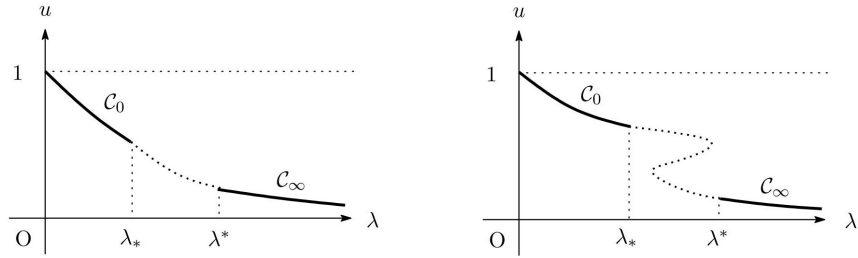} 
		  \caption{Possible positive solution sets in the case where  $\beta_\Omega<1$.} 
		\label{figbetaleq1b} 
		  \end{center}
		    \end{figure}

\begin{rem} 
\strut
\begin{enumerate}\setlength{\itemsep}{0.1cm} 



\item For $(\lambda, u_\lambda)\in \mathcal{C}_0$ with $0\leq \lambda \leq \lambda_{\ast}$ in \eqref{C_0}, we present similar results as those in assertions (i-c). 
Indeed, $\lambda \mapsto u_\lambda$ is decreasing for $0<\lambda \leq \lambda_{\ast}$; if $0<\lambda_1<\lambda_2$ with the condition that $\lambda_1\leq\lambda_{\ast}< \lambda_2$, then $u_{\lambda_1}>u$ in $\overline{\Omega}$ for a positive solution $u$ of \eqref{p} with $\lambda=\lambda_2$. 

\item It is an open question to get the global uniqueness for a positive solution of \eqref{p} for all $\lambda>0$, i.e., $\lambda_{\ast}=\lambda^{\ast}$. In this case, $[0,\infty)\ni \lambda \mapsto u_\lambda$ is decreasing. 

\item For uniqueness and stability analysis of positive solutions for large parameters in nonlinear elliptic problems, we refer to \cite{Da1984, Wi1985, Wib1985, Li1985, Da1986, MiSu1997, GM2004, HS2006, HS2013, HS2016}. 

\end{enumerate}
\end{rem}

Our second main result is the counterpart of Theorem \ref{thm0}(III-i) and (III-iii) for the case where $\beta_\Omega =1$ and $pq>1$. 

\begin{theorem} \label{thm:=1<pq}  
Assume that $\beta_\Omega=1$ and $pq>1$.  Then, problem \eqref{p} possesses a positive solution $u_\lambda$ for every $\lambda > 0$ such that $u_\lambda>0$ in $\overline{\Omega}$, which satisfies that  
\begin{align} \label{tophOm}
u_\lambda \longrightarrow 0 \quad\mbox{ and } \ \  
\frac{u_\lambda}{\| u_\lambda \|} \longrightarrow \phi_\Omega \quad \mbox{ in } H^1(\Omega) 
\quad \mbox{ as } \lambda \to \infty. 
\end{align}
\end{theorem}

\begin{rem}
\strut
\begin{enumerate}\setlength{\itemsep}{0.1cm} 
\item The existence assertion holds for any $p>1$; thus, so does assertion \eqref{tophOm} (see Remark \ref{rem:p3.2}). 

\item Similarly as in Theorem \ref{thm0}(III-iii), assertion \eqref{tophOm} is valid if we assume a positive solution $(\lambda, u_\lambda)$ (which may take zero value somewhere on $\partial \Omega$) of \eqref{p} with $\lambda \rightarrow \infty$. 
\end{enumerate}
\end{rem}


Our third main result is the counterpart of Theorem \ref{thm0}(III-iv), (IV-i), and (IV-iii) for the case where $\beta_\Omega=1$. 

\begin{theorem}  \label{thm:nobif}  
Assume that $\beta_\Omega = 1$. Then, the following three assertions hold. 
\begin{enumerate} \setlength{\itemsep}{0.2cm} 

\item If $pq\leq1$, then $\overline{\lambda}<\infty$ where $\overline{\lambda}>0$ is defined by \eqref{ovlam}. 

\item If $pq\neq 1$, then the positive solution set 
$\{ (\lambda, u)\}$ of \eqref{p} does not meet the trivial line $\{ (\lambda, 0) : \lambda>0 \}$ in the topology of $H^1(\Omega)$ (nor $C(\overline{\Omega})$).  

\item If $pq \geq 1$, then it does not meet $\{ (0,0) \}$ in the topology of $H^1(\Omega)$ (nor $C(\overline{\Omega})$).

\end{enumerate}
\end{theorem}

Theorem \ref{thm:nobif} provides a guess (Remark \ref{rem:compo}) for the global extension of the $C^\infty$ positive solution curve $\mathcal{C}_0$ given by \eqref{C_0} in the case where $\beta_\Omega=1$ and $pq\leq1$. 

\begin{rem}  \label{rem:compo}
Assume that $\beta_\Omega = 1$ and $pq\leq1$. Let $\widetilde{\mathcal{C}}_0=\{ (\lambda, u) \} \subset [0, \infty)\times C(\overline{\Omega})$ be the {\it component} (maximal, closed, and connected subset) of nonnegative solutions of \eqref{p} that includes $\mathcal{C}_0$. From Theorems \ref{thm0}(I) and Theorem \ref{thm:nobif}(i), $\widetilde{\mathcal{C}}_0\setminus \{ (0,1)\} \subset \{ (\lambda,u) \in [0,\infty) \times C(\overline{\Omega}) : \lambda \leq \overline{\lambda}, \ 
u<1 \mbox{ in } \overline{\Omega} \}$. 
If we suppose that 
$\Gamma_0:= \left( \widetilde{\mathcal{C}}_0 \setminus \{ (0,1)\} \right) \cap \left\{ (\lambda,0), (0,u) : \lambda\geq 0, \, u\geq 0 \right\} \neq \emptyset$, 
then Theorem \ref{thm:nobif}(ii),(iii) show that $\Gamma_0 = \{ (0,0)\}$ and $\Gamma_0 \subset \{ (\lambda, 0) : \lambda \geq \Lambda_0 \}$ for some $\Lambda_0>0$ when $pq<1$ and $pq=1$, respectively. 
The existence of $\widetilde{\mathcal{C}_0}$ is still an open question. 
Suggested positive solution sets are illustrated in Figures \ref{figsuperc} and \ref{figcri}. 
\end{rem}

		 \begin{figure}[!htb]
	  	   \begin{center}
			\includegraphics[scale=0.18]{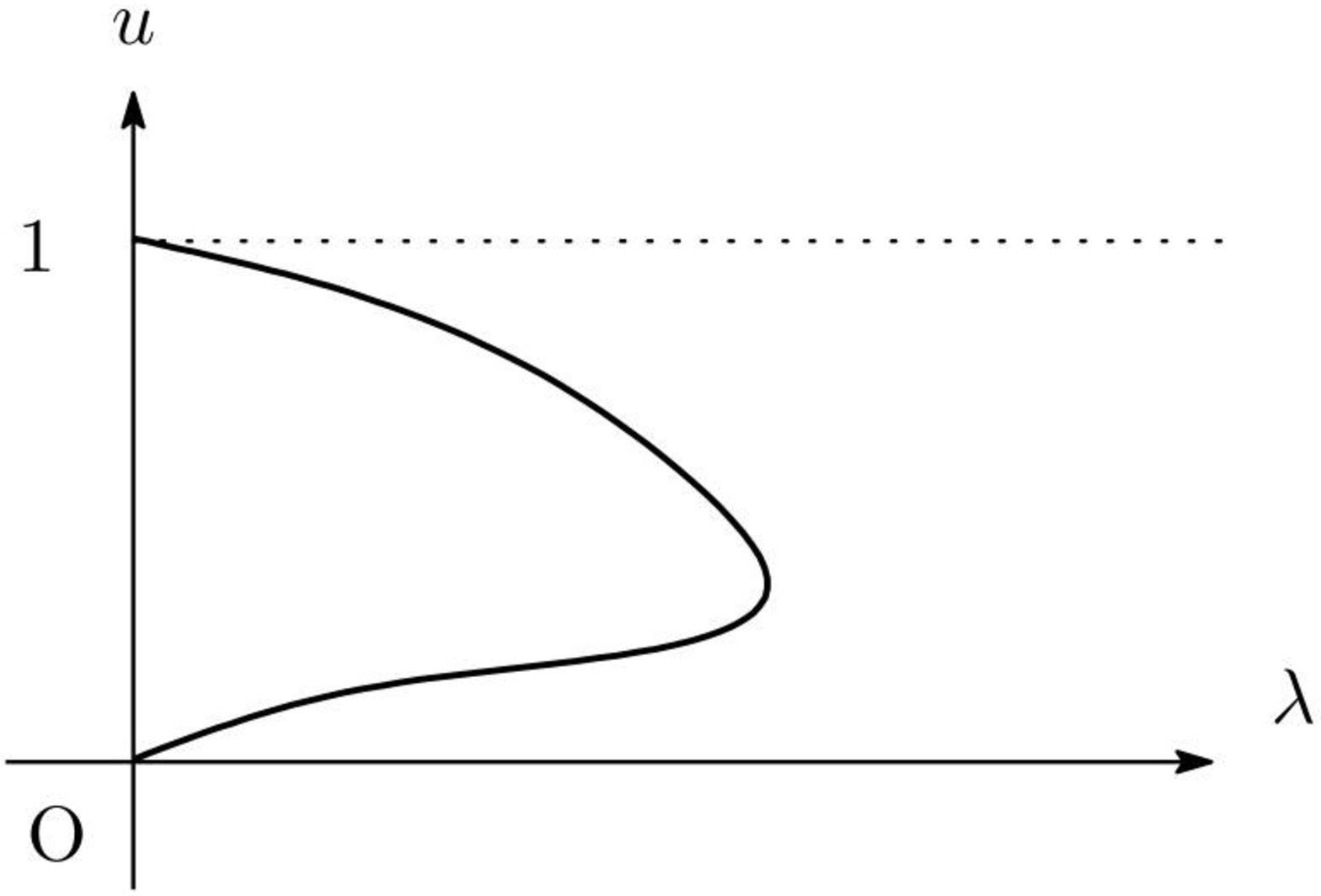} 
		  \caption{Suggested positive solution set in the case where  $\beta_\Omega =1$ and $pq<1$.} 
		\label{figsuperc} 
		  \end{center}
		    \end{figure}

		 \begin{figure}[!htb]
	  	   \begin{center}
			\includegraphics[scale=0.18]{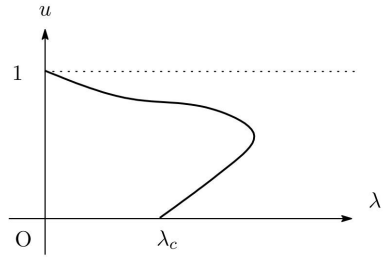}  
		  \caption{Suggested positive solution set in the case where $\beta_\Omega = 1$ and $pq=1$, and $\lambda_c\in \Gamma_0$.} 
		\label{figcri} 
		  \end{center}
		    \end{figure}

We conclude the Introduction by mentioning the stability of the trivial solution $u=0$. A linearized stability analysis does not work for $u=0$ because $u \mapsto u^q$ is not differentiable at $u=0$. 
Instead, by the construction of suitable sub- and supersolutions of \eqref{p}, we try to employ the Lyapunov stability criterion \cite[Chapter 5]{Pa92} on the basis of the monotone iteration method, which is developed in Section \ref{sec:trista}.


\begin{description}
\item[Notation] 
$\| \cdot\|$ denotes the usual norm of $H^1(\Omega)$. 
$u_n \rightharpoonup u_\infty$ means that $u_n$ weakly converges to $u_\infty$ in $H^1(\Omega)$. 
$H^1_0(\Omega)=\{ u\in H^1(\Omega) : u=0 \mbox{ on } \partial\Omega\}$. 
$\int_{\Omega} f dx$ for $f \in L^1(\Omega)$ and $\int_{\partial\Omega}g d\sigma$ for $g\in L^1(\partial\Omega)$ are simply written as $\int_{\Omega}f$ and $\int_{\partial\Omega}g$, respectively. 
$|\cdot|$ represents both the Lebesgue measure in $\Omega$ and the surface measure on $\partial\Omega$. 
\end{description}

The remainder of this paper is organized as follows. Sections \ref{sec:ener} and  \ref{sec:sub} are devoted to the preparation for the proofs of Theorems \ref{thm:leq1}, \ref{thm:=1<pq} and \ref{thm:nobif}. In Section \ref{sec:ener}, we develop the energy method for the energy functional associated with \eqref{p}. In Section \ref{sec:sub}, we use the sub- and supersolution method to prove existence and positivity results for positive solutions of \eqref{p}. We give proofs for Theorems \ref{thm:leq1} and \ref{thm:=1<pq} in Section \ref{sec:sub}. In Section \ref{sec:nobif}, we prove Theorem \ref{thm:nobif}. Section \ref{sec:trista} is devoted to a stability analysis of the trivial solution $u=0$, which is based on Lemma \ref{prop:sub} and Theorem \ref{thm:trista}.

\section{Energy method} 

\label{sec:ener}

Let 
\[
E(u)=\int_\Omega (|\nabla u|^2 - u^2), \quad u \in H^1(\Omega);  
\]
then, the next lemma is used several times in the following arguments.  

\begin{lem} \label{lem1719}
Let $\{ u_n \}\subset H^1(\Omega)$ satisfy $E(u_n)\leq0$, $u_n\rightharpoonup u_\infty$, and $u_n \rightarrow u_\infty$ in $L^2(\Omega)$. 
Then, $u_\infty\neq 0$ if $\| u_n\|\geq C$ for some $C>0$. 
\end{lem}

\begin{proof} 
By the weak lower semicontinuity, $E(u_\infty)\leq \varliminf_n E(u_n)\leq \varlimsup_n E(u_n) \leq0$. 
If $u_\infty=0$, then $\| u_n\|\rightarrow 0$, as desired. 
\end{proof}

We start by proving the following two propositions, which provide the asymptotic profile of a positive solution of \eqref{p} as $\lambda\to \infty$. 
It is understood that $u_{\mathcal{D}}=0$ if $\beta_\Omega=1$. 

\begin{prop} \label{prop:asympt}
Assume that $\beta_\Omega \leq1$. Let $(\lambda_n,u_n)$ be a positive solution of \eqref{p} with $\lambda_n \to\infty$. Then, $u_n \rightarrow u_{\mathcal{D}}$ in $H^1(\Omega)$. 
\end{prop}

\begin{proof}
We first assume that $\beta_\Omega < 1$. 
Because $u_n<1$ in $\overline{\Omega}$, 
we substitute $u=\varphi = u_n$ into \eqref{def} to deduce that 
\begin{align} 
\int_{\Omega}|\nabla u_n|^2 
& = \int_{\Omega} \biggl( u_n^2-u_n^{p+1} \biggr)-\lambda_n \int_{\partial\Omega}u_n^{q+1} \label{ener0756} \\ 
&\leq \int_{\Omega} u_n^2\leq |\Omega|;    \nonumber 
\end{align}
thus, $\| u_n\|$ is bounded. Immediately, up to a subsequence, $u_n \rightharpoonup u_\infty\geq0$, $u_n \rightarrow u_\infty$ in $L^2(\Omega)$ and $L^2(\partial\Omega)$, and $u_n \rightarrow u_\infty$ 
a.e.\ in $\Omega$ for some $u_\infty \in H^1(\Omega)$. We then infer that 
\begin{align} \label{unq+1:2} 
\int_{\partial\Omega}u_n^{q+1} = 
\frac{1}{\lambda_n}\left\{ -\int_{\Omega}|\nabla u_n|^2 +  \int_{\Omega} \left( u_n^2-u_n^{p+1} \right) \right\}
\leq \frac{1}{\lambda_n}\int_{\Omega}u_n^2 \longrightarrow 0, 
\end{align}
which implies that $\int_{\partial\Omega}u_\infty^{q+1}=0$; thus, $u_\infty\in H^1_0(\Omega)$. From \eqref{def} 
with $(\lambda,u)=(\lambda_n, u_n)$, it follows that 
\begin{align*}
\int_{\Omega} \biggl( \nabla u_n \nabla \varphi - u_n \varphi + u_n^p \varphi \biggr) = 0, \quad \varphi \in H^1_0(\Omega).     
\end{align*}
Taking the limit, $u_\infty$ is a nonnegative solution of \eqref{Dp}, where we have used the Lebesgue dominated convergence theorem to deduce that $\int_{\Omega} u_n^p \varphi \rightarrow \int_{\Omega} u_\infty^p \varphi$. 

Then, we verify that $u_\infty \neq 0$. Since $E(u_n)\leq0$, the weak lower semicontinuity means that  
\begin{align*}
E(u_\infty)\leq \varliminf_{n\to \infty}E(u_n)\leq \varlimsup_{n\to \infty} E(u_n)\leq0.    
\end{align*}
If $u_\infty=0$, then it follows that $\| u_n \|\rightarrow 0$. 
Here, we may assume that $u_n \rightarrow 0$ a.e.\ in $\Omega$.  
Say that $w_n=\frac{u_n}{\| u_n \|}$; then, up to a subsequence, 
$w_n \rightharpoonup w_\infty\geq0$, $w_n \rightarrow w_\infty$ in $L^2(\Omega)$ 
and $L^2(\partial\Omega)$, and $w_n \rightarrow w_\infty$ a.e.\ in $\Omega$. Since $\| w_n \|=1$, we deduce that $w_\infty \neq 0$ using Lemma \ref{lem1719}. 
However, we observe from \eqref{unq+1:2} that 
\begin{align*}
\int_{\partial\Omega} w_n^{q+1} \leq \frac{1}{\lambda_n}\int_{\Omega} w_n^2 \| u_n  \|^{1-q}\longrightarrow 0. 
\end{align*}
This implies that $w_\infty \in H^1_0(\Omega)$. 
From \eqref{def} with $(\lambda,u)=(\lambda_n, u_n)$, we see that 
\begin{align*}
\int_\Omega \left( \nabla w_n \nabla \varphi - w_n \varphi + w_n \varphi \, u_n^{p-1} \right) = 0, \quad \varphi \in H^1_0(\Omega). 
\end{align*}
Taking the limit, $\int_\Omega (\nabla w_\infty \varphi - w_\infty \varphi) = 0$, where we have used the Lebesgue dominated convergence theorem to obtain that 
\begin{align*}
\int_\Omega \left\vert w_n \varphi \, u_n^{p-1} \right\vert \leq 
\left( \int_\Omega w_n^2 \right)^{\frac{1}{2}} \left( \int_\Omega \varphi^2 u_n^{2(p-1)} \right)^{\frac{1}{2}} \leq C \left( \int_\Omega \varphi^2 u_n^{2(p-1)} \right)^{\frac{1}{2}} 
\longrightarrow 0. 
\end{align*}
This implies that $w_\infty$ is a nontrivial nonnegative solution of the problem
\begin{align*}
\begin{cases}
-\Delta w = w & \mbox{ in } \Omega, \\
w=0 & \mbox{ on } \partial \Omega. 
\end{cases}    
\end{align*}
Thus, we deduce that $\beta_\Omega = 1$, which contradicts the assumption. 
The assertion that $u_\infty\geq0$ and $u_\infty\neq 0$ means that 
$u_\infty$ is the unique positive solution $u_{\mathcal{D}}$ of \eqref{Dp} by the strong maximum principle. 

It remains to show that $u_n \rightarrow u_\infty$ in $H^1(\Omega)$. 
Observing that $E(u_\infty)+\int_\Omega u_\infty^{p+1}=0$ and 
$E(u_n) \leq - \int_\Omega u_n^{p+1}$, we deduce that 
\begin{align*}
E(u_\infty)\leq \varliminf_{n\to \infty}E(u_n)\leq \varlimsup_{n\to \infty}E(u_n) \leq - \varlimsup_{n\to \infty}\int_\Omega u_n^{p+1} = - \int_\Omega u_\infty^{p+1} = E(u_\infty),  
\end{align*}
where we have used the Lebesgue dominated convergence theorem again. 
Thus, $E(u_n)\rightarrow E(u_\infty)$, i.e., $\| u_n \| \rightarrow \| u_\infty\|$. Since $u_n \rightharpoonup u_\infty$, the desired assertion follows. 

Next, we assume that $\beta_\Omega = 1$. Then, 
$u_\infty$ is a nonnegative solution of \eqref{Dp}, and indeed $u_\infty=0$ because $\beta_\Omega = 1$ (\cite{BO86}). Thus, $\| u_n \| \rightarrow 0$, as desired. 
\end{proof}

In the case where $\beta_\Omega=1$, we observe that $E(u)\geq 0$ for $u\in H^1_0(\Omega)$. Indeed, we note that 
\begin{align*}
\left\{ u\in H^{1}_{0}(\Omega) : E(u)=0 \right\} = \langle \phi_\Omega \rangle := 
\{ s \phi_\Omega : s \in \mathbb{R} \}. 
\end{align*}
We then investigate the asymptotic profile of a positive solution $(\lambda_n,u_n)$ of \eqref{p} with $\| u_n \|\rightarrow 0$. 

\begin{prop} \label{prop:unprofi}
Assume that $\beta_\Omega=1$. Let $(\lambda_n,u_n)$ be a positive solution of \eqref{p} such that $\lambda_n \geq \underline{\lambda}$ for some $\underline{\lambda}>0$ and $\| u_n \| \rightarrow 0$.  
Then, we obtain that 
$\frac{u_n}{\| u_n \|} \rightarrow \phi_\Omega$ in $H^1(\Omega)$. 
\end{prop}

\begin{proof}
Say that $w_n = \frac{u_n}{\| u_n\|}$ and, up to a subsequence, $w_n \rightharpoonup w_\infty\geq0$, and $w_n \rightarrow w_\infty$ in $L^2(\Omega)$ and $L^2(\partial\Omega)$ for some $w_\infty \in H^1(\Omega)$. From \eqref{ener0756} it follows that 
\begin{align*} 
\underline{\lambda}\int_{\partial\Omega} u_n^{q+1} \leq 
\int_\Omega u_n^2. 
\end{align*}
We use the condition $\| u_n\| \rightarrow 0$ to infer that 
\begin{align*}
\int_{\partial\Omega}w_n^{q+1}\leq \frac{\| u_n\|^{1-q}}{\underline{\lambda}}\int_\Omega w_n^2 \longrightarrow 0; 
\end{align*}
thus, $\int_{\partial\Omega}w_\infty^{q+1}=0$, i.e., $w_\infty \in H^1_0(\Omega)$.

Since $\beta_\Omega=1$ and $E(w_n)\leq0$, the weak lower semicontinuity means that 
\begin{align*}
0\leq E(w_\infty)\leq \varliminf_{n\to \infty} E(w_n) \leq \varlimsup_{n\to \infty} E(w_n)\leq0, 
\end{align*}
implying that $E(w_n)\rightarrow E(w_\infty)=0$, i.e., 
$\| w_n\| \rightarrow \| w_\infty\|$. 
Since $w_n \rightharpoonup w_\infty$, we deduce that $w_n \rightarrow w_\infty$ in $H^1(\Omega)$ and $w_\infty = \phi_\Omega$ with $\| \phi_\Omega\|=1$. 
Finally, because $\phi_\Omega$ is unique, the desired conclusion follows. 
\end{proof}

\begin{rem} \label{rem:behave}
If we construct a positive solution $(\lambda_n, u_n)$ of \eqref{p} 
such that $u_{n}>0$ in $\overline{\Omega}$ without using \eqref{subcp}, 
then Propositions \ref{prop:asympt} and \ref{prop:unprofi} remain valid for all $p>1$. 
\end{rem}

For further analysis of the asymptotic behavior of a positive solution $(\lambda_n,u_n)$ of \eqref{p} 
with the condition that $\lambda_n\geq \underline{\lambda}$ and $\| u_n \| \rightarrow 0$, the orthogonal decomposition $H^1(\Omega) = \langle \phi_\Omega \rangle \oplus V$ using $\langle \phi_\Omega \rangle$ is useful, 
where $V$ denotes the orthogonal complement of $\langle \phi_\Omega \rangle$ that is given explicitly as 
\begin{align*}
V = \left\{ v \in H^1(\Omega) : \int_\Omega \biggl(\nabla v \nabla \phi_\Omega + v \phi_\Omega \biggr) = 0 \right\}. 
\end{align*}
Note that $\langle \phi_\Omega \rangle$ and $V$ are both closed subspaces of $H^1(\Omega)$ and $\|u\|$ is equivalent to $|s|+\|v\|$ for $u=s\phi_\Omega + v \in H^1(\Omega)=\langle \phi_\Omega\rangle\oplus V$. 

Using the orthogonal decomposition, 
\begin{align} \label{vn1557}
u_n = s_n \phi_\Omega + v_n \in \langle \phi_\Omega \rangle \oplus V 
\end{align}
for a positive solution $(\lambda_n,u_n)$ of \eqref{p} such that $\lambda_n\geq \underline{\lambda}$ for some $\underline{\lambda}>0$ and $\| u_n \|\rightarrow 0$ (under the assumption of Proposition \ref{prop:unprofi}). Since $\frac{u_n}{\| u_n\|} \rightarrow \phi_\Omega$ in $H^1(\Omega)$, it follows that 
\begin{align}
&\frac{s_n}{\|u_n\|}\rightarrow 1, \label{sn1145}\\
&\frac{\| v_n\|}{\| u_n\|}\rightarrow 0, \label{vn1145}\\
&\frac{\| v_n\|}{s_n} \rightarrow 0. \label{vnsn1145}
\end{align}
Because of \eqref{sn1145}, we may assume that $s_n>0$. Note that $v_n\geq 0$ on $\partial\Omega$ because $\phi_\Omega=0$ on $\partial\Omega$. 
We then deduce the following result, which plays a crucial role in the proof of Theorem \ref{thm:nobif}.

\begin{lem}  \label{prop:E} 
Assume that $\beta_\Omega=1$. Let $\{ v_n\}$ be as introduced by \eqref{vn1557}. Then, there exists $c>0$ such that 
\begin{align} \label{Evn1911}
E(v_n) + c \int_{\partial\Omega} v_n^{q+1}\leq 0 \quad\mbox{ for sufficiently large $n$,  }
\end{align}
provided that one of the following conditions is satisfied. 
\begin{enumerate} \setlength{\itemsep}{0.2cm} 
    \item[(a)] $pq<1$,
    \item[(b)] $pq=1$ and $\lambda_n \to \infty$,
    \item[(c)] $pq>1$ and $\lambda_n$ is bounded above. 
\end{enumerate}
\end{lem}

\begin{proof}
Substituting $u_n=s_n\phi_\Omega + v_n$ into \eqref{ener0756}, we deduce that 
\begin{align} \label{eq1211}
2s_n \left(\int_\Omega \nabla \phi_\Omega \nabla v_n - \phi_\Omega v_n \right) + E(v_n) + 
\int_\Omega (s_n\phi_\Omega + v_n)^{p+1} + \lambda_n \int_{\partial\Omega} v_n^{q+1} = 0. 
\end{align}
Using the divergence theorem, 
\begin{align} \label{vn1524}
\int_\Omega \phi_\Omega v_n = \int_\Omega -\Delta \phi_\Omega v_n = \int_\Omega \nabla \phi_\Omega \nabla v_n + \int_{\partial\Omega} \biggl( -\frac{\partial \phi_\Omega}{\partial \nu} \biggr)  v_n;  
\end{align}
thus, \eqref{eq1211} implies that
\begin{align*}
-2s_n \int_{\partial\Omega} \biggl( -\frac{\partial \phi_\Omega}{\partial \nu} \biggr) 
v_n + E(v_n) + \int_\Omega (s_n\phi_\Omega + v_n)^{p+1} + \lambda_n \int_{\partial\Omega} v_n^{q+1} = 0.      
\end{align*}
It follows that 
\begin{align} \label{En1034}
E(v_n) + \frac{\lambda_n}{2} \int_{\partial\Omega} v_n^{q+1} + I_n \leq 0 
\end{align}
with 
\begin{align} \label{In1034}
I_n = \frac{\lambda_n}{2}\int_{\partial\Omega} v_n^{q+1} - 2s_n \int_{\partial\Omega} \left( -\frac{\partial \phi_\Omega}{\partial \nu} \right)v_n. 
\end{align}
Once we verify that 
\begin{align} \label{In1338}
I_n \geq 0 \quad \mbox{ for sufficiently large $n$, }
\end{align}
we obtain \eqref{Evn1911} and complete the proof. 
To verify \eqref{In1338}, we use the test function $\varphi = \phi_\Omega$ to deduce that 
\begin{align} \label{unphOm1502}
\int_\Omega \biggl(\nabla u_n \nabla \phi_\Omega - u_n \phi_\Omega + u_n^{p} \phi_\Omega \biggr) = 0.    
\end{align}
Substituting $u_n = s_n \phi_\Omega + v_n$ into  
\eqref{unphOm1502} and combining \eqref{vn1524} with \eqref{unphOm1502} provide 
\begin{align} \label{vnsn1033}
\int_{\partial\Omega} \left(-\frac{\partial \phi_\Omega}{\partial \nu} \right) \frac{v_n}{s_n^p} = \int_\Omega \left( 
\phi_\Omega + \frac{v_n}{s_n} \right)^{p} \phi_\Omega. 
\end{align}

We then consider either case (a) or (b). From \eqref{vnsn1145}, we deduce that  
\begin{align*}
\int_\Omega \left( 
\phi_\Omega + \frac{v_n}{s_n} \right)^{p} \phi_\Omega 
\longrightarrow \int_\Omega \phi_\Omega^{p+1} > 0. 
\end{align*}
Taking into account \eqref{bdry1513}, we may derive from \eqref{vnsn1033} that 
\begin{align*}
c s_n^p \leq \int_{\partial\Omega} v_n
\end{align*}
for some $c>0$. By H\"older's inequality, it follows that 
\begin{align} \label{c41554}
c s_n^{p} \leq \int_{\partial\Omega}v_n \leq |\partial \Omega|^{\frac{q}{q+1}}\left( \int_{\partial\Omega}v_n^{q+1} \right)^{\frac{1}{q+1}}. 
\end{align}
Combining \eqref{In1034} with \eqref{c41554} and using H\"older's inequality, 
there exist $c, \tilde{c}>0$ such that 
\begin{align*}
I_n&\geq \frac{\lambda_n}{2} \int_{\partial\Omega}v_n^{q+1} - c\, s_n \left(\int_{\partial\Omega} v_n^{q+1}\right)^{\frac{1}{q+1}}\\ 
&=\left\{ \frac{\lambda_n}{2}\left( \int_{\partial\Omega} v_n^{q+1} \right)^{\frac{q}{q+1}} - c\, s_n \right\}
\left( \int_{\partial\Omega} v_n^{q+1} \right)^{\frac{1}{q+1}} \\ 
&\geq \left\{ \tilde{c}\, \lambda_n \, s_n^{pq} - c\, s_n \right\}\left( \int_{\partial\Omega} v_n^{q+1} \right)^{\frac{1}{q+1}} \\ 
&= s_n^{pq}\left\{ \tilde{c} \lambda_n - c\, s_n^{1-pq} \right\}\left( \int_{\partial\Omega} v_n^{q+1} \right)^{\frac{1}{q+1}}. 
\end{align*}
Since $s_n \to 0$ from \eqref{sn1145}, assertion \eqref{In1338} follows.

We next consider case (c) and verify \eqref{In1338}. By combining \eqref{In1034} with \eqref{vnsn1033}, it follows from \eqref{In1034} that 
\begin{align} \label{In1057}
I_n = s_n^{p+1} \biggl\{ \frac{\lambda_n}{2}\int_{\partial\Omega} \frac{v_n^{q+1}}{s_n^{p+1}} - 2 \int_\Omega \biggl( \phi_\Omega + \frac{v_n}{s_n}\biggr)^p \phi_\Omega \biggr\}. 
\end{align}
Furthermore, we use the test function $\varphi = 1$ in \eqref{def} to infer that
\begin{align*}
-\int_\Omega u_n + \int_\Omega u_n^p + \lambda_n \int_{\partial\Omega}u_n^q = 0.      
\end{align*}
Substituting $u_n = s_n\phi_\Omega + v_n$, 
\begin{align*}
-\int_\Omega \biggl(\phi_\Omega + \frac{v_n}{s_n} \biggr) + s_n^{p-1}\int_\Omega \biggl( \phi_\Omega + \frac{v_n}{s_n}\biggr)^p + \int_{\partial\Omega} \frac{\lambda_n v_n^q}{s_n}=0,    
\end{align*}
which implies that 
\begin{align*}
\int_{\partial\Omega} \frac{\lambda_n v_n^q}{s_n} \longrightarrow \int_\Omega \phi_\Omega >0, 
\end{align*}
where we have used condition \eqref{vnsn1145}. Thus, we may deduce that 
\begin{align*}
c\frac{s_n}{\lambda_n} \leq \int_{\partial\Omega} v_n^{q} 
\end{align*}
for some $c>0$. Using H\"older's inequality, we deduce that 
\begin{align} \label{vnq+1058}
c \left( \frac{s_n}{\lambda_n} \right)^{\frac{q+1}{q}} \leq \int_{\partial\Omega}v_n^{q+1}
\end{align}
for some $c>0$. Combining \eqref{In1057} with \eqref{vnq+1058}, 
\begin{align*}
I_n\geq s_n^{p+1} \biggl\{ c\,  s_n^{\frac{1}{q}-p}\lambda_n^{-\frac{1}{q}}
-2\int_\Omega \biggl( \phi_\Omega + \frac{v_n}{s_n}\biggr)^p \phi_\Omega \biggr\}
\end{align*}
for some $c>0$. We observe that 
\begin{align*}
& s_n^{\, \frac{1}{q}-p} \rightarrow \infty, \\
& \lambda_n^{-\frac{1}{q}} \geq c \quad \mbox{ by some constant $c>0$, } \\ 
& \int_\Omega \left( \phi_\Omega + \frac{v_n}{s_n}\right)^p \phi_\Omega \rightarrow \int_\Omega \phi_\Omega^{p+1} > 0. 
\end{align*}
Thus, assertion \eqref{In1338} follows. 
\end{proof}

We conclude this section with the establishment of the uniqueness and stability results for a positive solution of \eqref{p} in the case where $\beta_\Omega<1$.  

\begin{prop} \label{prop:unista}
Assume that $\beta_\Omega < 1$. Then, there exists $\underline{\Lambda}>0$ such that 
if $\lambda > \underline{\Lambda}$, then the following {\rm two} assertions hold:
\begin{enumerate} \setlength{\itemsep}{0.2cm} 

\item Problem \eqref{p} has at most one positive solution. 

\item A positive solution $u$ of \eqref{p} satisfying $u>0$ in $\overline{\Omega}$ is asymptotically stable.

\end{enumerate}
\end{prop}

\begin{proof}
We recall that the unique positive solution $u_{\mathcal{D}}$ of \eqref{Dp} is asymptotically stable, i.e., 
\begin{align} \label{gam1D0922}
\gamma_{1, \mathcal{D}} = \inf\biggl\{ \int_\Omega \left( |\nabla \varphi|^2 - \varphi^2 + p u_{\mathcal{D}}^{p-1}\varphi^2 \right): \varphi \in H^1_0(\Omega), \ \int_{\Omega}\varphi^2=1 \biggr\}>0.     
\end{align}

(i) Assume to the contrary that problem \eqref{p} has two distinct positive solutions 
$(\lambda_n, u_n)$ and $(\lambda_n, v_n)$ with  $\lambda_n\to \infty$. Note that $u_n, v_n<1$ in $\overline{\Omega}$. We may assume that $u_n, v_n \rightarrow u_{\mathcal{D}}$ in $H^1(\Omega)$ and a.e.\ in $\Omega$. 
The difference $w_n = u_n-v_n$ (may change sign) allows that
\begin{align} \label{wn1409}
\int_\Omega \biggl( \nabla w_n \nabla \varphi - w_n \varphi \biggr) + \int_{\Omega} \biggl\{  (v_n+w_n)^p - v_n^p \biggr\} \varphi 
+ \lambda_n \int_{\Gamma_n} \frac{(v_n+w_n)^q - v_n^q}{w_n} \, w_n \varphi = 0 
\end{align}
for $\varphi \in H^1(\Omega)$, where $\Gamma_n = \{ x\in \partial\Omega : w_n(x)\neq 0 \}$. 
Note that $\| w_n\| \rightarrow 0$, and $w_n \rightarrow 0$ a.e.\ in $\Omega$. 

Substituting $\varphi = w_n$ into \eqref{wn1409}, the mean value theorem shows the existence of $\theta_n \in (0,1)$ such that 
\begin{align*}
E(w_n) + \int_{\Omega} p(v_n + \theta_n w_n)^{p-1}w_n^2 + \lambda_n \int_{\Gamma_n} \frac{(v_n+w_n)^q - v_n^q}{w_n} w_n^2 = 0;  
\end{align*}
thus, $\psi_n = \frac{w_n}{\| w_n\|}$ implies that 
\begin{align} \label{psim1436}
E(\psi_n) + \int_{\Omega} p(v_n + \theta_n w_n)^{p-1}\psi_n^2 + \lambda_n \int_{\Gamma_n} \frac{(v_n+w_n)^q - v_n^q}{w_n} \psi_n^2 = 0. 
\end{align}
From $\|\psi_n \|=1$, we infer that up to a subsequence, $\psi_n \rightharpoonup \psi_\infty$, 
$\psi_n \rightarrow \psi_\infty$ in $L^2(\Omega)$ and $L^2(\partial\Omega)$ for some $\psi_\infty \in H^1(\Omega)$. Then, we claim that $\psi_\infty \in H^1_0(\Omega)$ and $\psi_\infty \neq 0$. From \eqref{psim1436}, we deduce that  
\begin{align*}
\lambda_n \int_{\Gamma_n} \frac{(v_n+w_n)^q - v_n^q}{w_n} \psi_n^2 \leq \int_\Omega \psi_n^2 \leq 1. 
\end{align*}
We observe that 
\begin{align*}
\frac{(v_n+w_n)^q - v_n^q}{w_n} \geq q \quad\mbox{ on } \Gamma_n 
\end{align*}
because $u_n, v_n <1$ in $\overline{\Omega}$. It follows that 
$\lambda_n q \int_{\partial\Omega} \psi_n^2\leq 1$. Passing to the limit, we deduce that $\psi_\infty \in H^1_0(\Omega)$. Indeed, $\psi_\infty\neq 0$ by Lemma \ref{lem1719}.  

Then, we assert in \eqref{psim1436} that 
\begin{align*}
\int_{\Omega} p(v_n + \theta_n w_n)^{p-1}\psi_n^2 \longrightarrow \int_\Omega pu_{\mathcal{D}}^{p-1} \psi_\infty^2.     
\end{align*}
Indeed, we use 
\begin{align*}
\int_{\Omega} p(v_n + \theta_n w_n)^{p-1}\psi_n^2 
= \int_\Omega p(v_n + \theta_n w_n)^{p-1} \psi_\infty^2 + \int_\Omega  p(v_n + \theta_n w_n)^{p-1} (\psi_n^2 - \psi_\infty^2). 
\end{align*}
Since $v_n \rightarrow u_{\mathcal{D}}$ and $w_n \rightarrow 0$ a.e.\ in $\Omega$ and $u_n, v_n<1$ in $\Omega$, the Lebesgue dominated convergence theorem shows that 
\begin{align*}
\int_\Omega p(v_n + \theta_n w_n)^{p-1} \psi_\infty^2 \longrightarrow \int_\Omega pu_{\mathcal{D}}^{p-1}\psi_\infty^2. 
\end{align*}
Using the fact $\int_\Omega |\psi_n^2 - \psi_\infty^2| \rightarrow 0$ yields 
\begin{align*}
\int_\Omega  p(v_n + \theta_n w_n)^{p-1} (\psi_n^2 - \psi_\infty^2) \longrightarrow 0,      
\end{align*}
as desired. 

Then, the weak lower semicontinuity allows us to deduce from \eqref{psim1436} that 
\begin{align*}
\int_\Omega \biggl(|\nabla \psi_\infty|^2 - \psi_\infty^2 + pu_{\mathcal{D}}^{p-1} \psi_\infty^2 \biggr) \leq \varliminf_{n\to \infty} \biggl( E(\psi_n) + \int_{\Omega} p(v_n + \theta_n w_n)^{p-1}\psi_n^2 \biggr) \leq 0, 
\end{align*}
which contradicts $\psi_\infty \in H^1_0(\Omega)$ and $\psi_\infty \neq 0$ in view of \eqref{gam1D0922}.

(ii) On the basis of \eqref{gam0956}, we claim that $\gamma_1 > 0$ for sufficiently large $\lambda>0$. Assume by contradiction that a positive solution $(\lambda_n,u_n)$ of \eqref{p} with the condition that $\lambda_n \to \infty$ and $u_n>0$ in $\overline{\Omega}$ satisfies $\gamma_n:=\gamma_{1,n}(\lambda_n,u_n)\leq0$. This means that
\begin{align} \label{gamn0925}
\int_\Omega \biggl( |\nabla \varphi_n|^2 - \varphi_n^2 + pu_n^{p-1} \varphi_n^2 \biggr) + \lambda_n q \int_{\partial \Omega}u_n^{q-1} \varphi_n^2 = \gamma_n \leq0, 
\end{align}
where $\varphi_n:=\varphi_{1,n}$, normalized as
\begin{align} \label{norm0910}
\int_{\Omega}\varphi_{n}^2 + \int_{\partial\Omega}\varphi_{n}^2 = 1. 
\end{align}
Because $\int_\Omega |\nabla \varphi_n|^2\leq \int_\Omega \varphi_n^2\leq 1$ from \eqref{gamn0925} and \eqref{norm0910}, $\| \varphi_n \|$ is bounded, which implies that up to a subsequence, $\varphi_n \rightharpoonup \varphi_\infty$, $\varphi_n \rightarrow \varphi_\infty$ in $L^2(\Omega)$ and $L^2(\partial\Omega)$, and $\varphi_n \rightarrow \varphi_\infty$ a.e.\ in $\Omega$ for some $\varphi_\infty \in H^1(\Omega)$. 
Since $u_n^{q-1}\geq1$ from Theorem \ref{thm0}(I), assertion \eqref{gamn0925} gives us 
\begin{align*}
\lambda_n q \int_{\partial \Omega} \varphi_n^2 \leq \int_{\Omega} \varphi_n^2 \leq1.  
\end{align*}
Passing to the limit, $\varphi_\infty = 0$ on $\partial\Omega$; thus, $\varphi_\infty \in H^1_0(\Omega)$. From \eqref{norm0910},  $\int_\Omega \varphi_\infty^2=1$ is also deduced.

By the weak lower semicontinuity, we derive from \eqref{gamn0925} that 
\begin{align} 
\int_\Omega \biggl( |\nabla \varphi_\infty|^2 - \varphi_\infty^2 + pu_{\mathcal{D}}^{p-1} \varphi_\infty^2 
\biggr) \leq \varliminf_{n\to \infty} \int_\Omega \biggl( |\nabla \varphi_n|^2 - \varphi_n^2 + pu_{n}^{p-1} \varphi_n^2 \biggr) \leq0 \label{wlsc0926}  
\end{align}
Indeed, on the basis of the facts that $u_n \rightarrow u_{\mathcal{D}}$ in $H^1(\Omega)$ and $u_n<1$ in $\overline{\Omega}$ (see Theorem \ref{thm0}(I) and Proposition \ref{prop:asympt}), the Lebesgue dominated convergence theorem shows that 
\[
\int_\Omega u_{n}^{p-1} \varphi_n^2 
= \int_\Omega u_{n}^{p-1} \varphi_\infty^2 +\int_\Omega u_{n}^{p-1} (\varphi_n^2 - \varphi_\infty^2) \longrightarrow \int_\Omega u_{\mathcal{D}}^{p-1} \varphi_\infty^2, 
\]
where we have used that $u_n \rightarrow u_{\mathcal{D}}$ a.e.\ in $\Omega$. 
Assertion \eqref{wlsc0926} contradicts $\varphi_\infty \in H^1_0(\Omega)$ and $\int_\Omega \varphi_\infty^2=1$ in view of \eqref{gam1D0922}.  \end{proof}

\begin{rem} \label{rem:stab}

If we construct a positive solution $u$ of \eqref{p} such that $u>0$ in $\overline{\Omega}$ without using \eqref{subcp}, then assertion (ii) of Proposition  \ref{prop:unista} remains valid for all $p>1$. 
\end{rem}

\section{Sub- and supersolutions} 

\label{sec:sub}


Consider the case where $\beta_\Omega < 1$ or where $\beta_\Omega = 1$ and $pq>1$. Then, we first construct small positive subsolutions of \eqref{p} and use them to establish an {\it a priori} lower bound for positive solutions $(\lambda, u)$ of \eqref{p} satisfying $u>0$ in $\overline{\Omega}$. For a fixed $\tau > 0$ and with a parameter $\varepsilon > 0$, we set
\[
\phi_\varepsilon (x) = \varepsilon (\phi_\Omega (x) + \varepsilon^\tau), \quad x \in \overline{\Omega},  
\]
which implies that $\phi_\varepsilon \in C^{2+\theta}(\overline{\Omega})$ and $\phi_\varepsilon > 0$ in $\overline{\Omega}$. 

We then use $\phi_\varepsilon$ to formulate the following {\it a priori} lower bound for positive solutions $u>0$ in $\overline{\Omega}$ of \eqref{p}.

\begin{lem} \label{prop:sub}
Assume that $\beta_\Omega < 1$ or that $\beta_\Omega = 1$ and $pq>1$. Let $\tau  > \frac{1-q}{q}$ when $\beta_\Omega < 1$, and let $\frac{1-q}{q} < \tau < p-1$ when $\beta_\Omega = 1$ and $pq>1$. 
Then, for $\Lambda > 0$ there exists $\overline{\varepsilon}=\overline{\varepsilon}(\tau, \Lambda)>0$ such that $\phi_{\varepsilon}$ is a subsolution of \eqref{p}, provided that 
$\lambda \in [0, \Lambda]$ and $\varepsilon \in (0, \overline{\varepsilon}]$. 
Furthermore, $u \geq \phi_{\overline{\varepsilon}}$ in $\overline{\Omega}$ for a positive solution $u > 0$ in $\overline{\Omega}$ of \eqref{p} with $\lambda \in [0,\Lambda]$. Here, $\overline{\varepsilon}$ does not depend on $\lambda \in [0,\Lambda]$. 
\end{lem}
\begin{proof}
We only consider the case where $\beta_\Omega = 1$ and $pq>1$. The case where $\beta_\Omega < 1$ is proved similarly. 
First, we verify the former assertion. We take $0<\varepsilon\leq 1$ and then use the condition $p-\tau -1>0$ to deduce that  
\begin{align*}
-\Delta \phi_\varepsilon -\phi_\varepsilon +\phi_\varepsilon^p 
\leq \varepsilon^{1+\tau}\biggl\{ -1 + \varepsilon^{p-\tau -1} \left( 1+\max_{\overline{\Omega}}\phi_\Omega \right)^p \biggr\}\leq 0 \quad\mbox{ in } \Omega  
\end{align*}
if $\varepsilon > 0$ is small. For $\Lambda > 0$ we use \eqref{bdry1513} and the condition $\tau > \frac{1-q}{q}$ to deduce that 
\begin{align*}
\frac{\partial \phi_\varepsilon}{\partial \nu} + \lambda \phi_\varepsilon^q 
\leq -c_1 \varepsilon + \lambda \varepsilon^{(1+\tau)q} 
\leq \varepsilon (-c_1 + \Lambda \varepsilon^{q+\tau q -1}) \leq 0 \quad 
\mbox{ on } \partial\Omega 
\end{align*}
if $0<\varepsilon\leq (c_1/\Lambda)^{1/(q+\tau q-1)}$. The desired assertion follows. 

Next, we argue by contradiction to verify the latter assertion. Assume by contradiction that $u\not\geq \phi_{\overline{\varepsilon}}$ in $\overline{\Omega}$ for some positive solution $u>0$ in $\overline{\Omega}$ with $\lambda \in [0, \Lambda]$. Because $\varepsilon \mapsto \phi_\varepsilon$ is increasing and $\phi_\varepsilon \rightarrow 0$ uniformly in $\overline{\Omega}$, we can take $\varepsilon_1\in (0, \overline{\varepsilon})$ such that 
\begin{align}
\left\{ \begin{array}{l}
u\geq \phi_{\varepsilon_1} \ \ \mbox{ in } \ \overline{\Omega}, \\
u(x_1) = \phi_{\varepsilon_1}(x_1) \ \ \mbox{ for some } \ x_1 \in \overline{\Omega}. 
\end{array} \right.  \label{bdry1251}
\end{align}
Take $c>0$ such that $u, \phi_{\varepsilon_1}\geq c$ in $\overline{\Omega}$; then, choose $K>0$ sufficiently large so that $f_K(t)=Kt + t-t^p$ is increasing for $t \in \left[ 0, \, \max_{\overline{\Omega}}u \right]$ and $M>0$ sufficiently large so that $M-\Lambda q c^{q-1}>0$. We use the subsolution $\phi_{\varepsilon_1}$ (not a positive solution of \eqref{p}) to deduce that 
\begin{align} \label{KOm}
(-\Delta + K)(u-\phi_{\varepsilon_1})\geq f_K(u) - f_K(\phi_{\varepsilon_1}) \geq 0 \ (\mbox{and} \not\equiv 0) \quad\mbox{ in } \ \Omega,
\end{align}
and for $x \in \partial\Omega$ satisfying $u>\phi_{\varepsilon_1}$, 
\begin{align}
\biggl( \frac{\partial}{\partial \nu} + M \biggr) (u-\phi_{\varepsilon_1})
&\geq -\lambda u^q + \lambda \phi_{\varepsilon_1}^q + M(u-\phi_{\varepsilon_1}) \nonumber \\ 
&= \biggl( M - \lambda \frac{u^q - \phi_{\varepsilon_1}^q}{u - \phi_{\varepsilon_1}} \biggr)(u-\phi_{\varepsilon_1}) \nonumber \\
& \geq (M - \Lambda q c^{q-1})(u-\phi_{\varepsilon_1})>0. \label{KpOm}
\end{align}
Thus, the strong maximum principle and boundary point lemma are applicable to infer that $u-\phi_{\varepsilon_1}>0$ in $\overline{\Omega}$, which contradicts \eqref{bdry1251}. 
\end{proof}

On the basis of Lemma \ref{prop:sub}, we construct minimal and maximal positive solutions of \eqref{p} as follows. 
\begin{prop} \label{prop:maxmin} 
Assume that $\beta_\Omega < 1$ or that $\beta_\Omega = 1$ and $pq>1$. Then, problem \eqref{p} has a minimal positive solution $\underline{u}_{\lambda} \in C^{2+\theta}(\overline{\Omega})$ and a maximal positive solution $\overline{u}_\lambda \in C^{2+\theta}(\overline{\Omega})$ for each $\lambda > 0$ such that $0<\underline{u}_\lambda \leq \overline{u}_\lambda$ in $\overline{\Omega}$, meaning that any positive solution $u$ of \eqref{p} 
with the condition that $u>0$ in $\overline{\Omega}$ 
satisfies $\underline{u}_\lambda \leq u \leq \overline{u}_\lambda$ in $\overline{\Omega}$. Moreover, both $\underline{u}_\lambda$ and $\overline{u}_{\lambda}$ are weakly stable, i.e., $\gamma_1(\lambda, \underline{u}_\lambda), \gamma_1(\lambda,\overline{u}_\lambda)\geq 0$. 
\end{prop}

\begin{proof}
It is clear that $(\lambda, 1)$ is a supersolution of \eqref{p}. Choose $\varepsilon_0>0$ such that $\phi_{\varepsilon_0}\leq 1$ in $\overline{\Omega}$; then, Lemma \ref{prop:sub} states that $(\lambda,  \phi_{\varepsilon_0})$ is a subsolution of \eqref{p}. By Theorem \ref{thm0}(I) and Lemma \ref{prop:sub}, a positive solution $u>0$ in $\overline{\Omega}$ of \eqref{p} implies that  $\phi_{\varepsilon_0}\leq u\leq 1$ in $\overline{\Omega}$. 
Thus, this proposition is a direct consequence of applying \cite[(2.1)Theorem]{Am76N} and \cite[Proposition 7.8]{Am76}. 
\end{proof}

\begin{rem} \label{rem:p3.2}
In view of the construction, Lemma \ref{prop:sub} and Proposition \ref{prop:maxmin} remain valid for any $p>1$; therefore, Propositions \ref{prop:asympt}, \ref{prop:unprofi}, \ref{prop:unista}(ii) hold with any $p>1$ for the positive solution $(\lambda_n,u_n)$ and $(\lambda,u)$ of 
\eqref{p} with $\lambda_n\to \infty$ that are constructed by Proposition \ref{prop:maxmin} (see Remarks \ref{rem:behave} and \ref{rem:stab}). 
\end{rem}

In the case where $\beta_\Omega<1$, Propositions \ref{prop:unista} and \ref{prop:maxmin} ensure the existence of a unique positive solution $u(\lambda)$ of \eqref{p} for $\lambda > \underline{\Lambda}$, which is asymptotically stable.  Using the implicit function theorem provides us with the following result. 

\begin{cor} \label{cor:curve}
Assume that $\beta_\Omega < 1$. Then, $\{ (\lambda, u(\lambda)) : \lambda > \underline{\Lambda} \}$ is a $C^\infty$ curve, i.e., $\lambda \mapsto u(\lambda) \in C^{2+\theta}(\overline{\Omega})$ is $C^\infty$. Moreover, it is decreasing, i.e., $u(\lambda_1)>u(\lambda_2)$ in $\overline{\Omega}$ for $\lambda_2>\lambda_1>\underline{\Lambda}$. 
\end{cor}

\begin{proof}
We verify the first assertion. Let $(\lambda_0,u(\lambda_0))$ be the unique positive solution of \eqref{p} for $\lambda_0 >\underline{\Lambda}$. Since $\gamma_1(\lambda_0, u(\lambda_0))>0$, the implicit function theorem applies at $(\lambda_0,u_0)$; then, we deduce, thanks to the uniqueness, that $\{ (\lambda, u(\lambda)) : \lambda_1 < \lambda \leq \lambda_2 \}$ is a $C^\infty$ curve for $\lambda_1<\lambda_0<\lambda_2$ such that $u(\lambda)$ is asymptotically stable. The implicit function theorem applies again at $(\lambda_2,u(\lambda_2))$; then, the curve is continued until $\lambda = \lambda_3>\lambda_2$. Repeating the same procedure, the curve is continued to $\lambda=\infty$ thanks to the {\it a priori} upper and lower bounds (Theorem \ref{thm0}(I) and Lemma \ref{prop:sub}), as desired.

We next verify the second assertion. 
If $\lambda_1<\lambda_2$, then $u(\lambda_1)$ is a supersolution of \eqref{p} for $\lambda = \lambda_2$. By Lemma \ref{prop:sub}, it is possible to construct a subsolution $\phi_\varepsilon$ of \eqref{p} for $\lambda = \lambda_2$ such that $0<\phi_\varepsilon \leq u(\lambda_1)$ in $\overline{\Omega}$. The sub- and supersolution method applies, and problem \eqref{p} has a positive solution $u$ for $\lambda=\lambda_2$ such that $\phi_\varepsilon \leq u \leq u(\lambda_1)$ in $\overline{\Omega}$, where $u\not\equiv u(\lambda_1)$. Thanks to Proposition \ref{prop:unista}(i), we obtain $u=u(\lambda_2)$. The desired assertion follows by using the strong maximum principle and boundary point lemma (as developed in \eqref{KOm} and \eqref{KpOm}). \end{proof}

We conclude this section by employing the weak sub- and supersolution method \cite{LS06} to show {\it global} strong positivity for a positive solution of \eqref{p} in the case where $\beta_\Omega<1$.  

\begin{prop} \label{pop:posi}
Assume that $\beta_\Omega < 1$. Then, a positive solution $(\lambda, u)$ of \eqref{p} satisfies that $u>0$ in $\overline{\Omega}$. 
\end{prop}

\begin{proof}
Assume by contradiction that problem \eqref{p} possesses a positive solution $(\lambda_0, u_0)$ for $\lambda_0 >0$ such that $u_0=0$ somewhere on $\partial\Omega$. 
Let $\lambda=\lambda_1 >\max(\lambda_0, \underline{\Lambda})$, for which problem \eqref{p} has at most one positive solution by Proposition \ref{prop:unista}(i).  
Then, $u_0$ is a weak supersolution of \eqref{p} for $\lambda = \lambda_1$. Indeed, 
\begin{align*}
0 &=\int_\Omega \biggl( \nabla u_0 \nabla \varphi - u_0 \varphi + u_0^p \varphi \biggr) + \lambda_0 \int_{\partial\Omega} u_0^{q} \varphi \\ 
& \leq \int_\Omega \biggl( \nabla u_0 \nabla \varphi - u_0 \varphi + u_0^p \varphi \biggr) + \lambda_1 \int_{\partial\Omega} u_0^{q} \varphi, \quad \varphi\in H^1(\Omega) \mbox{ and } \varphi\geq 0. 
\end{align*}
We next construct a weak subsolution of \eqref{p} for $\lambda=\lambda_1$ that is smaller than or equal to $u_0$. Note that $u_0 \in C^\theta(\overline{\Omega})$ and $u_0>0$ in $\Omega$. From $\beta_\Omega < 1$, it follows by the continuity and monotonicity of $\beta_\Omega$ with respect to $\Omega$ that we can choose a subdomain $\Omega_1 \Subset \Omega$ with smooth boundary $\partial\Omega_1$ such that $\beta_{\Omega_1}<1$. Then, we deduce that
\begin{align} \label{bddblw}
u_0\geq c \quad\mbox{ in } \ \overline{\Omega_1}
\end{align}
for some $c>0$. We also deduce that if $\varepsilon>0$ is sufficiently small, then 
\begin{align*} 
-\Delta (\varepsilon \phi_{\Omega_1}) \leq \varepsilon \phi_{\Omega_1} - (\varepsilon \phi_{\Omega_1})^p \quad \mbox{ in } \Omega_1,  
\end{align*}
where $\phi_{\Omega_1}$ is a positive eigenfunction associated with $\beta_{\Omega_1}$. Consequently, the divergence theorem is applied to $\int_{\Omega_{1}} -\Delta (\varepsilon \phi_{\Omega_1}) \varphi$ for $\varphi \in H^1(\Omega)$ and $\varphi \geq 0$ to deduce that 
\begin{align} \label{l:sub}
\int_{\Omega_1} \biggl( \nabla (\varepsilon \phi_{\Omega_1}) \nabla \varphi - (\varepsilon \phi_{\Omega_1}) \varphi + (\varepsilon \phi_{\Omega_1})^p \varphi \biggr) \leq 0.     
\end{align}
Define 
\begin{align*}
\Phi_\varepsilon = \left\{ \begin{array}{ll}
\varepsilon \phi_{\Omega_1} & \mbox{ in } \overline{\Omega_1},   \\
0 & \mbox{ in } \Omega \setminus \overline{\Omega_1}, 
\end{array}    \right.
\end{align*}
and $\Phi_\varepsilon \in H^1(\Omega) \cap C(\overline{\Omega})$. By virtue of \eqref{l:sub}, the linking technique \cite[Lemma I.1]{BL80} yields 
\begin{align*}
\int_\Omega \biggl( \nabla\Phi_\varepsilon \nabla \varphi - \Phi_\varepsilon \varphi + \Phi_\varepsilon^p \, \varphi \biggr) + \lambda_1 \int_{\partial\Omega} \Phi_\varepsilon^{q}\, \varphi 
= \int_{\Omega } \biggl( \nabla\Phi_\varepsilon \nabla \varphi - \Phi_\varepsilon \varphi + \Phi_\varepsilon^p \, \varphi \biggr) \leq0. 
\end{align*}
Thanks to \eqref{bddblw}, we can take $\varepsilon>0$ such that $\Phi_\varepsilon \leq u_0$ in $\overline{\Omega}$, as desired.

The weak sub- and supersolution method \cite[Subsection 2.2]{LS06} 
is now applicable to deduce the existence of a positive solution $(\lambda_1, u_1)$ of \eqref{p} such that $\Phi_\varepsilon\leq u_1 \leq u_0$ in $\overline{\Omega}$. 
Particularly, $u_1=0$ somewhere on $\partial\Omega$ because so is 
$u_0$. However, this contradicts Proposition \ref{prop:maxmin} in view of the uniqueness. 
\end{proof}

\begin{proof}[Proof of Theorem \ref{thm:leq1}] 
The uniqueness assertion follows from Proposition \ref{prop:unista}(i). Assertions (i-a) and (i-b) follow from Propositions \ref{prop:unista}(ii) and \ref{prop:asympt}, respectively.  Assertion (i-c) is verified by Corollary \ref{cor:curve} and an analogous argument as in the proof of Corollary \ref{cor:curve}. Assertion (ii) follows from Proposition \ref{pop:posi}. 
\end{proof}

\begin{proof}[Proof of Theorem \ref{thm:=1<pq}] 
The existence part follows from Proposition \ref{prop:maxmin}. Assertion \eqref{tophOm} follows from Propositions \ref{prop:asympt} and \ref{prop:unprofi}. 
\end{proof}

\section{Proof of Theorem \ref{thm:nobif}}

\label{sec:nobif} 

This section is devoted to the proof of Theorem \ref{thm:nobif}. 

(i) We prove assertion (i). 
Assume by contradiction that problem \eqref{p} has a positive solution $(\lambda_n,u_n)$ with $\lambda_n \to \infty$. Then, Proposition \ref{prop:asympt} shows that $\| u_n \| \rightarrow 0$; thus, Proposition \ref{prop:unprofi} shows that $\frac{u_n}{\| u_n\|} \rightarrow \phi_\Omega$ in $H^1(\Omega)$. We apply Lemma \ref{prop:E}(a) and (b); then, for $u_n = s_n \phi_\Omega + v_n \in \langle \phi_\Omega \rangle \oplus V$ as in \eqref{vn1557}, we have \eqref{Evn1911} with \eqref{sn1145}--\eqref{vnsn1145}.

Observe from \eqref{vn1145} that $\| v_n\|\rightarrow 0$. 
Say that $\psi_n=\frac{v_n}{\| v_n\|}$; then, up to a subsequence, $\psi_n \rightharpoonup \psi_\infty$ ($\geq0$ on $\partial\Omega$), and $\psi_n \rightarrow \psi_\infty$ in $L^2(\Omega)$ and $L^2(\partial\Omega)$ for some $\psi_\infty \in H^1(\Omega)$. 
From \eqref{Evn1911}, it follows that 
\[
c\int_{\partial\Omega}\psi_n^{q+1}\leq -E(\psi_n)\| v_n\|^{1-q}  \longrightarrow 0, 
\]
so that $\int_{\partial\Omega} \psi_\infty^{q+1}=0$, i.e., $\psi_\infty\in H^1_0(\Omega)$. Lastly, we use the condition $E(\psi_n)\leq 0$ derived from \eqref{Evn1911} to follow the argument in the last paragraph of the proof of Proposition \ref{prop:unprofi}; then, we arrive at the contradiction $0\neq \psi_\infty \in \langle \phi_\Omega \rangle \cap V = \{ 0\}$. 

(ii) We verify assertion (ii). We remark that the convergences $(\lambda_n,u_n) \rightarrow (\lambda_\infty,0)$ with $\lambda_\infty\geq0$ in $\mathbb{R}\times H^1(\Omega)$ and $\mathbb{R}\times C(\overline{\Omega})$ are equivalent 
for a positive solution $(\lambda_n,u_n)$ of \eqref{p} with $\lambda_n>0$. This is verified by the bootstrap argument \cite[Lemma 3.3]{Um2022}. In fact, the proof of assertion (ii) is similar to that for assertion (i). Assume by contradiction that problem \eqref{p} has a positive solution $(\lambda_n, u_n)$ with the condition that $\lambda_n \rightarrow \lambda_\infty > 0$ and $\| u_n\|\rightarrow 0$. Lemma \ref{prop:E}(a) and (c) apply; then, we arrive at a contradiction.  

(iii) To verify assertion (iii), we prove the following {\it three} auxiliary lemmas. Say that $U_n=\lambda_n^{-\frac{1}{1-q}}u_n$.

\begin{lem} \label{lem1809}
There exists $C>0$ such that $\| U_n\|\leq C$ for a positive solution $(\lambda_n,u_n)$ of \eqref{p} with $\lambda_n>0$ satisfying that $(\lambda_n,u_n)\rightarrow (0,0)$ in $\mathbb{R}\times H^1(\Omega)$. 
\end{lem}

\begin{proof}
Assume by contradiction that $\| U_n\|\rightarrow \infty$. Say that $w_n = \frac{U_n}{\| U_n\|}$; then, up to a subsequence, $w_n \rightharpoonup w_\infty\geq0$, and $w_n \to w_\infty$ in $L^{2}(\Omega)$ and $L^2(\partial\Omega)$ for some $w_\infty \in H^1(\Omega)$.  Since $E(w_n)\leq0$, Lemma \ref{lem1719} provides $w_\infty\neq 0$. 

Recall that $(\lambda,U)=(\lambda_n,U_n)$ satisfies 
\begin{align} \label{Upro1758} 
\int_\Omega \biggl( \nabla U \nabla \varphi - U \varphi + \lambda^{\frac{p-1}{1-q}} U^{p}\varphi \biggr) + \int_{\partial\Omega} U^{q}\varphi = 0, \quad \varphi \in H^1(\Omega). 
\end{align}
Using the test function $\varphi=1$ in \eqref{Upro1758}, we deduce that 
\begin{align*}
\int_\Omega U_n = \lambda_n^{\frac{p-1}{1-q}}\int_\Omega U_n^p + \int_{\partial\Omega}U_n^q 
= \int_\Omega u_n^{p-1}U_n + \int_{\partial\Omega}U_n^q, 
\end{align*}
implying 
\begin{align} \label{wn1624}
\int_\Omega w_n = \int_\Omega u_n^{p-1}w_n + \int_{\partial\Omega}w_n^q \| U_n\|^{q-1}.  
\end{align}
We may assume that $u_n \rightarrow 0$ a.e.\ in $\Omega$, and since $u_n<1$ in $\overline{\Omega}$, we deduce that 
\[
\int_\Omega u_n^{p-1}w_n = \int_\Omega u_n^{p-1}w_\infty + \int_\Omega u_n^{p-1}(w_n - w_\infty) \longrightarrow 0, 
\]
by applying the Lebesgue dominated convergence theorem and using the condition $w_n \rightarrow w_\infty$ in $L^2(\Omega)$. Then, passing to the limit in \eqref{wn1624} yields $\int_\Omega w_\infty = 0$, i.e., $w_\infty=0$, which is a contradiction.
\end{proof}

\begin{lem} \label{lem1819}
Assume that $\beta_\Omega = 1$. Then, there is no positive solution $U$ of \eqref{Upro1758} for $\lambda=0$.
\end{lem}

\begin{proof}
If it exists, then from \eqref{Upro1758} with $\lambda=0$ and $\varphi=1$, 
it follows that $U>0$ on $\Gamma \subset \partial\Omega$ with $|\Gamma|>0$, implying $\int_{\partial\Omega} \frac{\partial \phi_\Omega}{\partial \nu} U<0$. We use the test function $\varphi = \phi_\Omega$ to deduce that 
\begin{align*}
\int_\Omega \biggl( \nabla U \nabla \phi_\Omega - U\phi_\Omega \biggr) = 0. 
\end{align*}
However, the divergence theorem leads us to the contradiction
\begin{align*}
\int_\Omega \phi_\Omega U = \int_\Omega -\Delta \phi_\Omega U = \int_\Omega \nabla \phi_\Omega \nabla U - \int_{\partial\Omega} \frac{\partial \phi_\Omega}{\partial \nu} U > \int_\Omega \nabla \phi_\Omega \nabla U.
\end{align*}
\end{proof}

\begin{lem} \label{lem1813}
Assume that $\beta_\Omega = 1$ and $pq\geq1$. Then, there exists $C>0$ such that $\| U_n \|\geq C$ for a positive solution $(\lambda_n, U_n)$ of \eqref{Upro1758} with $\lambda_n \to 0^{+}$. 
\end{lem}

\begin{proof}
Assume by contradiction that $(\lambda_n,U_n) \rightarrow (0,0)$ in $\mathbb{R}\times H^1(\Omega)$ for a positive solution $(\lambda_n, U_n)$ of \eqref{Upro1758}. Say that $w_n = \frac{U_n}{\| U_n\|}$; then, up to a subsequence, $w_n \rightharpoonup w_\infty\geq0$, $w_n \rightarrow w_\infty$ in $L^{p+1}(\Omega)$ and $L^2(\partial\Omega)$ for some $w_\infty \in H^1(\Omega)$. From \eqref{Upro1758} with $(\lambda, U)=(\lambda_n, U_n)$ and $\varphi = U_n$, it follows that 
\begin{align} \label{Un1819}
\int_\Omega \biggl( |\nabla U_n|^2 - U_n^2 + \lambda_n^{\frac{p-1}{1-q}}U_n^{p+1} \biggr) + \int_{\partial\Omega} U_n^{q+1} = 0.     
\end{align}
We then deduce that $\int_{\partial\Omega} w_n^{q+1}\leq \int_{\Omega} w_n^2 \| U_n \|^{1-q} \rightarrow 0$; thus, $\int_{\partial\Omega}w_\infty^{q+1}=0$, i.e., $w_\infty \in H^1_0(\Omega)$. We also deduce from \eqref{Un1819} that $E(w_n)=\int_\Omega (|\nabla w_n|^2 - w_n^2)\leq0$. Thus, we derive that $w_n \rightarrow \phi_\Omega$ in $H^1(\Omega)$ using a similar argument as in the last paragraph of the proof of Proposition \ref{prop:unprofi}.

For a contradiction, we use the same strategy developed in the proof of assertion (i). To this end, we consider the orthogonal decomposition $U_n = s_n\phi_\Omega + v_n \in \langle \phi_\Omega \rangle \oplus V$ as in \eqref{vn1557}; then, we obtain \eqref{sn1145} to \eqref{vnsn1145} with $u_n$ replaced by $U_n$. 
As in the proof of Lemma \ref{prop:E},  
we deduce the following counterpart of \eqref{En1034} and \eqref{In1034} for \eqref{Un1819}:
\begin{align}
&E(v_n) + \frac{1}{2} \int_{\partial\Omega} v_n^{q+1} + J_n \leq 0, \quad\mbox{ with } \nonumber \\
&J_n = \frac{1}{2} \int_{\partial\Omega} v_n^{q+1}-2s_n\int_{\partial\Omega} \biggl( 
-\frac{\partial \phi_\Omega}{\partial \nu} \biggr) v_n. \label{Jn1522}
\end{align}
In the same spirit of Lemma \ref{prop:E} (\eqref{In1338}), we establish 
\begin{align}\label{Evn1824}
E(v_n) + \frac{1}{2} \int_{\partial\Omega} v_n^{q+1} \leq 0 \quad \mbox{ for sufficiently large $n$, } 
\end{align}
by verifying that 
\begin{align} \label{Jn}
J_n\geq0 \quad \mbox{ for sufficiently large $n$. }
\end{align}
Analogously to \eqref{vnsn1033}, we obtain 
\begin{align*}
\int_{\partial\Omega} \biggl( -\frac{\partial \phi_\Omega}{\partial \nu}\biggr) v_n = \lambda_n^{\frac{p-1}{1-q}} s_n^p \int_\Omega \left( 
\phi_\Omega + \frac{v_n}{s_n} \right)^{p} \phi_\Omega. 
\end{align*}
Using this assertion, we deduce from \eqref{Jn1522} that 
\begin{align} 
J_n = s^{p+1} \biggl\{ \frac{1}{2} \int_{\partial\Omega} \frac{v_n^{q+1}}{s^{p+1}} - 2\lambda_n^{\frac{p-1}{1-q}}  
\int_\Omega \left(\phi_\Omega + \frac{v_n}{s_n} \right)^{p} \phi_\Omega \biggr\}.  \label{Jn1253}
\end{align} 
Furthermore, we use the test function $\varphi = 1$ in \eqref{Upro1758} to obtain  
\begin{align*}
-\int_\Omega U_n + \lambda_n^{\frac{p-1}{1-q}} \int_\Omega U_n^{p} + \int_{\partial\Omega} U_n^{q} = 0. 
\end{align*}
Substituting $U_n = s_n\phi_\Omega + v_n$, 
\begin{align*}
-\int_\Omega \biggl( \phi_\Omega + \frac{v_n}{s_n} \biggr) + \lambda_n^{\frac{p-1}{1-q}} s_n^{p-1} \int_\Omega \biggl( \phi_\Omega + \frac{v_n}{s_n} \biggr)^p 
+ \int_{\partial\Omega}\frac{v_n^q}{s_n} = 0, 
\end{align*}
from which we use \eqref{vnsn1145} with $U_n$ to infer that 
\begin{align*}
\int_{\partial\Omega}\frac{v_n^q}{s_n}  \longrightarrow \int_\Omega \phi_\Omega >0.      
\end{align*}
Then, we may deduce that 
\begin{align*}
c s_n \leq \int_{\partial\Omega} v_n^q 
\end{align*}
for some $c>0$. By H\"older's inequality, we deduce that  
\begin{align*}
c s_{n}^{\frac{q+1}{q}}\leq \int_{\partial\Omega}v_{n}^{q+1}
\end{align*}
for some $c>0$. We use this inequality to derive from \eqref{Jn1253} that  
\begin{align*}
J_n \geq s^{p+1} \biggl\{ cs_n^{\frac{1}{q}-p} - 2\lambda_n^{\frac{p-1}{1-q}}  
\int_\Omega \left(\phi_\Omega + \frac{v_n}{s_n} \right)^{p} \phi_\Omega \biggr\}  
\end{align*}
for some $c>0$; thus, \eqref{Jn} follows. Assertion \eqref{Evn1824} has been now established. 

We end the proof of this lemma. Observe from \eqref{vn1145} with $U_n$ that $\| v_n\|\rightarrow 0$. 
Then, we develop the same argument as in the second paragraph of the proof of assertion (i) to arrive at the same contradiction. 
\end{proof}

Employing the above lemmas, we then verify assertion (iii). Assume by contradiction that $(\lambda_n,u_n)\rightarrow (0,0)$ in $\mathbb{R}\times H^1(\Omega)$ for a positive solution $(\lambda_n,u_n)$ of \eqref{p} with $\lambda_n>0$. 
Then, $(\lambda_n, U_n)$ with $U_n = \lambda_n^{-\frac{1}{1-q}}u_n$ admits a positive solution of \eqref{Upro1758}. Since $U_n$ is bounded in $H^1(\Omega)$ by Lemma \ref{lem1809}, we deduce that up to a subsequence, 
$U_n \rightharpoonup U_\infty\geq0$, and $U_n \rightarrow U_\infty$ in $L^{p+1}(\Omega)$ and $L^2(\partial\Omega)$ for some $U_\infty\in H^1(\Omega)$. Thanks to Lemma \ref{lem1813}, we apply Lemma \ref{lem1719} to obtain $U_\infty\neq 0$.  

Furthermore, substituting $(\lambda,U)=(\lambda_n,U_n)$ into \eqref{Upro1758} and then taking the limit, we deduce that 
\begin{align*}
\int_\Omega \biggl( \nabla U_\infty \nabla \varphi - U_\infty\varphi \biggr) + \int_{\partial\Omega} U_\infty^q \varphi = 0.     
\end{align*} 
This implies that $U_\infty$ is a nonnegative solution of \eqref{Upro1758} for $\lambda=0$. 
Finally, Lemma \ref{lem1819} provides $U_\infty=0$, which is a contradiction.

The proof of Theorem \ref{thm:nobif} is complete.  \qed

\begin{rem}
Assertions (ii) and (iii) of Theorem \ref{thm:nobif} 
are derived also from Lemma \ref{prop:sub} when $\beta_\Omega = 1$ and 
$pq>1$ additionally if the positive solution is positive in $\overline{\Omega}$.   
\end{rem}

\section{Stability analysis of the trivial solution} 

\label{sec:trista}

In the last section, we consider the stability of the trivial solution $u=0$. 
It is worthwhile to mention that a linearized stability analysis does not work for $u=0$ because $u \mapsto u^q$ is not differentiable at $u=0$. 
The corresponding initial-boundary value problem is formulated as follows:
\begin{align}  \label{ibp}
\begin{cases}
\frac{\partial u}{\partial t}(t,x) = \Delta u + u-u^{p} & \mbox{ in } (0,\infty) \times \Omega, \\
\frac{\partial u}{\partial \nu}  = -\lambda u^q & \mbox{ on } (0,\infty)\times \partial\Omega, \\
u(0,x)=u_0(x)\geq0 & \mbox{ in } \overline{\Omega}.
\end{cases} 
\end{align}
We use the method of monotone iterations to determine the Lyapunov stability of the trivial solution $u=0$ (see \cite[Definition 5.1.1]{Pa92}).

When $\beta_\Omega<1$ or when $\beta_\Omega=1$ and $pq>1$, we observe from Lemma \ref{prop:sub} that $u=0$ is {\it unstable} in the following sense: 
for $u_0 \in C^2(\overline{\Omega})$ sufficiently small such that $u_0>0$ in $\overline{\Omega}$, the  positive solution $u(t,x)$ of \eqref{ibp} corresponding to the initial value $u_0$ moves away from $0$ as $t\to \infty$. 

When $\beta_\Omega > 1$, for $\varepsilon, \delta , \tau > 0$, we set 
\begin{align*}
\psi_{\delta, \varepsilon, \tau}(x) = \delta (\phi_\Omega (x)+\varepsilon)^{\tau}, \quad x\in \overline{\Omega}.     
\end{align*}
Let $\Omega_\rho:=\{ x\in \Omega : {\rm dist}(x,\partial\Omega)<\rho \}$ for $\rho>0$ be a tubular neighborhood of $\partial\Omega$. Then, by \eqref{bdry1513}, for $\rho_0>0$ small, we can choose a constant $c_3=c_3(\rho_0)>0$ such that $|\nabla \phi_\Omega |^2\geq c_3$ in $\Omega_{\rho}$ for $0<\rho\leq \rho_0$. If $0<\rho \leq \rho_0$, then there exists $c_4=c_4(\rho)>0$ such that $\phi_\Omega \geq c_4$ in $E_\rho:=\Omega \setminus \Omega_\rho$. 

The following result would then provide useful information about the stability of the trivial solution $u=0$.

\begin{theorem} \label{thm:trista}
Assume that $\beta_\Omega>1$. Then, for $\frac{1}{\beta_\Omega}<\tau < 1$ and $\varepsilon>0$ small, there exists $\delta_1>0$ such that $\psi_{\delta, \varepsilon, \tau}$ is a supersolution of \eqref{p} whenever $0<\delta\leq \delta_1$. 
\end{theorem}

\begin{proof}
We write $\psi_{\delta, \varepsilon, \tau}$ simply as $\phi_{\delta, \varepsilon}$.  
By direct computations, we obtain 
\begin{align} 
&\nabla \psi_{\delta, \varepsilon} = \delta \tau (\phi_\Omega + \varepsilon)^{\tau -1} \nabla \phi_\Omega, \label{1stpsi}\\
&\Delta \psi_{\delta, \varepsilon}=\delta \tau (\tau -1)(\phi_\Omega + \varepsilon)^{\tau -2} |\nabla \phi_\Omega|^2 + \delta \tau (\phi_\Omega + \varepsilon)^{\tau -1} \Delta \phi_\Omega.   \label{2ndpsi}
\end{align}
We see from \eqref{2ndpsi} that for $x \in \Omega_\rho$,
\begin{align*} 
\Delta \psi_{\delta, \varepsilon} + \psi_{\delta,\varepsilon} - \psi_{\delta, \varepsilon}^p 
&\leq \delta \tau (\tau -1)(\phi_\Omega + \varepsilon)^{\tau -2} |\nabla \phi_\Omega|^2 + \delta (\phi_\Omega + \varepsilon)^\tau \\
&=\delta (\phi_\Omega + \varepsilon)^{\tau -2} \left\{ -\tau (1-\tau)c_3 + \left( \varepsilon + \max_{\overline{\Omega_\rho}}\phi_\Omega \right)^2 \right\}. 
\end{align*}
We then find $0<\rho_1\leq \rho_0$ and $\varepsilon_1>0$ such that 
\begin{align*}
\left( \varepsilon + \max_{\overline{\Omega_{\rho_1}}}\phi_\Omega \right)^2 \leq \tau (1-\tau)c_3 \quad\mbox{ for } \ 0<\varepsilon\leq \varepsilon_1,  
\end{align*}
and then, 
\begin{align*}
-\Delta \psi_{\delta, \varepsilon}+\psi_{\delta, \varepsilon}-\psi_{\delta, \varepsilon}^p \leq0 \quad \mbox{ in } \Omega_{\rho_1}. 
\end{align*}

Let us fix $c_4=c_4(\rho_1)$, and let $0<\varepsilon\leq \varepsilon_1$. We also see from \eqref{2ndpsi} that for $x \in E_{\rho_1}$, 
\begin{align*}
\Delta \psi_{\delta, \varepsilon} + \psi_{\delta, \varepsilon} - \psi_{\delta, \varepsilon}^p 
&\leq \delta \tau (\phi_\Omega + \varepsilon)^{\tau -1}(-\beta_\Omega)\phi_\Omega + \delta (\phi_\Omega + \varepsilon)^\tau \\ 
& \leq \delta (\phi_\Omega + \varepsilon)^{\tau -1} \left\{ (1-\tau \beta_\Omega) c_4 + \varepsilon \right\}. 
\end{align*}
We then determine $0<\varepsilon_2\leq \varepsilon_1$ such that $(1-\tau \beta_\Omega) c_4 + \varepsilon_2 \leq 0$, and then, 
\begin{align*}    
-\Delta \psi_{\delta, \varepsilon_2}+\psi_{\delta, \varepsilon_2}-\psi_{\delta, \varepsilon_2}^p \leq0 \quad \mbox{ in } E_{\rho_1}. 
\end{align*}

Finally, using \eqref{bdry1513}, we see from \eqref{1stpsi} that
\begin{align*}
\frac{\partial \psi_{\delta, \varepsilon_2}}{\partial \nu} + \lambda \psi_{\delta, \varepsilon_2}^q \geq \delta^q (-\delta^{1-q} \tau \varepsilon_2^{\tau -1} c_2 + \lambda \varepsilon_2^{\tau q}) \geq0 \quad\mbox{ on } \partial\Omega, 
\end{align*}
if $0<\delta\leq \delta_1$ for some $\delta_1>0$. 

In summary, $\psi_{\delta, \varepsilon_2}$, $0<\delta\leq \delta_1$, is as desired. 
\end{proof}

From Theorem \ref{thm:trista}, it might be claimed that $u=0$ is {\it asymptotically stable} for the case where $\beta_\Omega>1$, meaning that for $u_0$ in the order interval 
$[0, \psi_{\delta_1, \varepsilon_2, \tau}]$, the positive solution $u(t,x)$ of \eqref{ibp} associated with $u_0$ tends to $0$ as $t\to \infty$. If this occurs, then Theorem \ref{thm0}(II) means that 
problem \eqref{ibp} is bistable with two nonnegative stable equilibria for $0<\lambda\leq \lambda_{\ast}$ (one is $u_\lambda$, and the other is $u=0$), which presents ecologically a {\it conditional persistence} strategy for the harvesting effort $\lambda$. 
However, the difficulty arises from the fact that the monotone iteration scheme does not work for \eqref{ibp} in the order interval $[0, \psi_{\delta_1, \varepsilon_2, \tau}]$ because $u\mapsto (-u^q)$ does not satisfy the one-sided Lipschitz condition \cite[(4.1.19)]{Pa92} for $u$ close to $0$. Rigorous verification of the claim is an open question.


\end{document}